\documentclass{amsart}
\usepackage{amssymb,amsmath}

\theoremstyle{plain}
\newtheorem{thm}{Theorem}[section]

\newtheorem{lem}[thm]{Lemma}

\newtheorem{cor}[thm]{Corollary}

\theoremstyle{definition}

\theoremstyle{remark}
\newtheorem{rem}[thm]{Remark}

\newtheorem*{remark*}{Remark}

\numberwithin{equation}{section}

\newcommand{\cF}{{\mathcal{F}}}

\newcommand{\cH}{{\mathcal{H}}}

\newcommand{\cS}{{\mathcal{S}}}

        \newcommand{\field}[1]{{\mathbb{#1}}}
        \newcommand{\NN}{\field{N}}

        \newcommand{\RR}{\field{R}}

\allowdisplaybreaks

\begin{document}
\title[Semiclassical spectral asymptotics
for a Schr\"odinger operator]{Semiclassical spectral asymptotics for
a two-dimensional magnetic Schr\"odinger operator: The case of
discrete wells}

\author{Bernard Helffer}

\address{D\'epartement de Math\'ematiques, B\^atiment 425, Univ
Paris-Sud et CNRS, F-91405 Orsay C\'edex, France}
\email{Bernard.Helffer@math.u-psud.fr}

\author{Yuri A. Kordyukov }
\address{Institute of Mathematics, Russian Academy of Sciences, 112 Chernyshevsky
str. 450008 Ufa, Russia} \email{yurikor@matem.anrb.ru}

\thanks{Y.K. is partially supported by the Russian Foundation of Basic
Research (grant 09-01-00389).}

\dedicatory{Dedicated to Misha Shubin on the occasion of his 65th
birthday}

\subjclass[2000]{35P20, 35J10, 47A75, 58J50, 81Q10}

\keywords{Spectral theory, Schr\"odinger operators, magnetic fields,
eigenvalue asymptotics}

\begin{abstract}
We consider a magnetic Schr\"odinger operator $H^h$, depending on
the semiclassical parameter $h>0$, on a two-dimensional Riemannian
manifold. We assume that there is no electric field. We suppose that
the minimal value $b_0$ of the magnetic field $b$ is strictly
positive, and there exists a unique minimum point of $b$, which is
non-degenerate. The main result of the paper is a complete
asymptotic expansion for the low-lying eigenvalues of the operator
$H^h$ in the semiclassical limit. We also apply these results to
prove the existence of an arbitrary large number of spectral gaps in
the semiclassical limit in the corresponding periodic setting.
\end{abstract}

 \maketitle

\section{Preliminaries and main results}
Let $ M$ be a compact oriented manifold of dimension $n\geq 2$
(possibly with boundary). Let $g$ be a Riemannian metric and $\bf B$
a real-valued closed 2-form on $M$. Assume that $\bf B$ is exact and
choose a real-valued 1-form $\bf A$ on $M$ such that $d{\bf A} = \bf
B$. Thus, one has a natural mapping
\[
u\mapsto ih\,du+{\bf A}u
\]
from $C^\infty_c(M)$ to the space $\Omega^1_c(M)$ of smooth,
compactly supported one-forms on $M$. The Riemannian metric allows
to define scalar products in these spaces and consider the adjoint
operator
\[
(ih\,d+{\bf A})^* : \Omega^1_c(M)\to C^\infty_c(M).
\]
A Schr\"odinger operator with magnetic potential $\bf A$ is defined
by the formula
\[
H^h = (ih\,d+{\bf A})^* (ih\,d+{\bf A}).
\]
Here $h>0$ is a semiclassical parameter. If $M$ has non-empty
boundary, we will assume that the operator $H^h$ satisfies the
Dirichlet boundary conditions.

We are interested in semiclassical asymptotics of the low-lying
eigenvalues of the operator $H^h$. This problem was studied in
\cite{FoHel1,miniwells,HM,HM01,HelMo3,HelMo4,HelSj7,LuPa2,Mat,MatUeki,Mont,PiFeSt,Ray2,Ray3}
(see \cite{Fournais-Helffer:book,luminy} for surveys).

In this paper, we study the problem in a particular situation. We
come back to the case considered in \cite{HM01}. We suppose that $M$
is two-dimensional. Then we can write ${\bf B}=b dx_g$, where $b\in
C^\infty(M)$ and $dx_g$ is the Riemannian volume form. Let
\[
b_0=\min_{x\in M}b(x).
\]
We furthermore assume that:
\begin{enumerate}
  \item $b_0>0$;
  \item there exist a unique point $x_0$, which belongs to the interior of $M$,
  $k\in \NN$ and $C>0$ such that for all $x$ in some neighborhood
  of $x_0$ the estimates hold:
\[
C^{-1}\, d(x,x_0)^2\leq b(x)-b_0 \leq C \, d(x,x_0)^2\,.
\]
\end{enumerate}

Denote
\[
a={\rm Tr}\left(\frac12 {\rm Hess}\,b(x_0)\right)^{1/2}, \quad
d=\det \left(\frac12 {\rm Hess}\,b(x_0)\right).
\]

Denote by $\lambda_0(H^h)\leq \lambda_1(H^h)\leq \lambda_2(H^h)\leq
\ldots$ the eigenvalues of the operator $H^h$ in $L^2(M)$.

\begin{thm}\label{t:main}
Under current assumptions, for any natural $j$, there exist $C_j>0$
and $h_j>0$ such that, for any $h\in (0,h_j]$,
\[
hb_0 +h^2\left[\frac{2d^{1/2}}{b_0}j+
\frac{a^2}{2b_0}\right]-C_jh^{19/8}\leq \lambda_j(H^h) \leq hb_0
+h^2\left[\frac{2d^{1/2}}{b_0}j+ \frac{a^2}{2b_0}\right]+C_jh^{5/2}.
\]
\end{thm}

In particular, we have lower and upper bounds for the groundstate
energy $\lambda_0(H^h)$:
\[
hb_0 +h^2\frac{a^2}{2b_0}-C_0h^{19/8}\leq \lambda_0(H^h) \leq hb_0
+h^2\frac{a^2}{2b_0}+C_0h^{5/2}, \quad h\in (0,h_0].
\]
and the asymptotics of the splitting between the groundstate energy
and the first excited state~:
\[
\lambda_1(H^h)- \lambda_0(H^h) \sim h^2 \frac{2d^{\frac 12}}{b_0}\,.
\]

The previous statement can be completed in the following way.

\begin{thm}\label{t:mainbis}
Under current assumptions, for any natural $j$, there exists a
sequence $(\alpha_{j,\ell})_{\ell\in \mathbb N}$, and for  any $N$, there exist $C_{j,N}>0$
and $h_{j,N}>0$ such that, for any $h\in (0,h_{j,N}]$,
\begin{equation}\label{e:mainbis}
|\lambda_j(H^h) - h \sum_{\ell=0}^N \alpha_{j,\ell} h^{\frac \ell
2}|\leq C_{j,N} h^{\frac{N+3}{2}}\,,
\end{equation}
with $\alpha_{j,0} = b_0$, $\alpha_{j,1}=0$, $\alpha_{j,2}=
\frac{2d^{1/2}}{b_0}j+ \frac{a^2}{2b_0}$\,.
\end{thm}
This theorem improves the result of \cite{HM01} which only gives  a
two-terms asymptotics for the ground state energy in the flat case.

The scheme of the proof is to first prove the weak version, given by
Theorem~\ref{t:main}, with $N=2$, permitting to determine $j_0$
disjoint intervals in which the first $j_0$ eigenvalues are
localized,  for $h$ small enough, and then to determine a complete
expansion of each eigenvalue lying in a given interval.

The paper is organized as follows. In Section~\ref{s:upper} we
construct approximate eigenfunctions of the operator $H^h$ with any
order of precision. This allows us to prove accurate upper bounds
for the $j$th eigenvalue of $H^h$. In Section~\ref{s:lower} we prove
a lower bound for the $j$th eigenvalue of $H^h$ and complete the
proofs of Theorems~\ref{t:main} and~\ref{t:mainbis}. In
Section~\ref{s:gaps} we consider the case when the magnetic field is
periodic. We combine the construction of approximate eigenfunctions
given in Section~\ref{s:upper} with the results of \cite{diff2006}
to prove the existence of arbitrary large number of gaps in the
spectrum of the periodic operator $H^h$ in the semiclassical limit.

We wish to thank the Erwin Schr\"odinger Institute in Vienna and the
organizers of the conference in honor of 65th birthday of Mikhail
Shubin ``Spectral Theory and Geometric Analysis'' in Boston for
their hospitality and support.

\section{Upper bounds}\label{s:upper}

\subsection{Approximate eigenfunctions: main result}
The purpose of this section is to prove the following accurate upper
bound for the eigenvalues of the operator $H^h$.

\begin{thm}\label{t:qmodes}
Under current assumptions, for any natural $j$ and $k$, there exists
a sequence $(\mu_{j,k,\ell})_{\ell\in \mathbb N}$ with
\[
\mu_{j,k,0} = (2k+1)b_0, \quad \mu_{j,k,1}=0,
\]
and
\[ \mu_{j,k,2}= (2j+1)(2k+1)\frac{d^{1/2}}{b_0}
+(2k^2+2k+1)\frac{t}{2b_0} +\frac{1}{2}(k^2+k) R(x_0)\,,
\]
where $R$ is the scalar curvature, and
\[
t={\rm Tr}\left(\frac12 {\rm Hess}\,b(x_0)\right),
\]
and for any $N$, there exist $\phi^h_{jkN}\in C^\infty(M)$,
$C_{jk,N}>0$ and $h_{jk,N}>0$ such that
\begin{equation}\label{e:orth}
(\phi^h_{j_1k_1N},\phi^h_{j_2k_2N}) =\delta_{j_1j_2}\delta_{k_1k_2}+
\mathcal O_{j_1,j_2,k_1,k_2}(h)
\end{equation}
and, for any $h\in (0,h_{jk,N}]$,
\begin{equation}\label{e:Hh}
\|H^h\phi^h_{jkN}- \mu_{jkN}^h \phi^h_{jkN}\|\leq
C_{jkN}h^{\frac{N+3}{2}}\|\phi^h_{jkN}\|,\,
\end{equation}
where
\[
\mu_{jkN}^h=h\sum_{\ell=0}^N \mu_{j,k,\ell} h^{\frac \ell 2}.
\]
\end{thm}

Since the operator $H^h$ is self-adjoint, using Spectral Theorem, we
immediately deduce the existence of eigenvalues near the points
$\mu_{jk}^h$.

\begin{cor}\label{c:dist}
For any natural $j$, $k$ and $N$, there exist $C_{jk,N}>0$ and
$h_{jk,N}>0$ such that, for any $h\in (0,h_{jk,N})$,
\[
{\rm dist}(\mu_{jkN}^h, {\rm Spec}(H^h))\leq
C_{jk,N}h^{\frac{N+3}{2}}.
\]
\end{cor}

\begin{rem}
The low-lying eigenvalues of the operator $H^h$, as $h\rightarrow
0$, are obtained by taking $k=0$ in Theorem~\ref{t:qmodes}.
Therefore, as an immediate consequence of Theorem~\ref{t:qmodes}, we
deduce that, for any natural $j$ and $N$, there exists $h_{j,N}>0$
such that, for any $h\in (0,h_{j,N}]$, we have
\[
\lambda_j(H^h)\leq \mu_{j0N}^h + C_{j0,N}h^{\frac{N+3}{2}}.
\]
In particular, this implies the upper bound in Theorem~\ref{t:main}.
\end{rem}

\begin{rem}
Our interest in the case of arbitrary $k$ in Theorem~\ref{t:qmodes}
is motivated by its importance for proving the existence of gaps in
the spectrum of the operator $H^h$ in the semiclassical limit. This
will be discussed in Section~\ref{s:gaps}.
\end{rem}

\begin{proof}[Proof of Theorem \ref{t:qmodes}]
The proof is long, so we will split it in different steps in the
next subsections.
\end{proof}

\subsection{Expanding operators in fractional powers of $h$}
The approximate eigenfunctions $\phi^h_{jk}\in C^\infty(M)$, which
we are going to construct, will be supported in a small neighborhood
of $x_0$. So, in a neighborhood of $x_0$, we will consider some
special local coordinate system with coordinates $(x,y)$ such that
$x_0$ corresponds to $(0,0)$. We will only apply our operator on
functions which are a product of cut-off functions with functions of
the form of linear combinations of terms like $h^\nu
w(h^{-1/2}x,h^{-1/2}y)$ with $w$ in $\cS(\RR^2)$. These functions
are consequently $O(h^\infty)$ outside a fixed neighborhood of
$(0,0)$. We will start by doing the computations formally in the
sense that everything is determined modulo $O(h^\infty)$, and any
smooth function will be replaced by its Taylor's expansion. It is
then easy to construct non formal approximate eigenfunctions.

First, we recall that in local coordinates $X=(X^1,X^2)=(x,y)$ on
$M$ the 1-form $\bf A$ is written as
\[
{\bf A}=A_1(X)\,dX_1+A_2(X)\,dX_2,
\]
the matrix of the Riemannian metric $g$ as
\[
g(X)=(g_{j\ell}(X))_{1\leq j,\ell\leq 2}
\]
and its inverse as
\[
g(X)^{-1}=(g^{j\ell}(X))_{1\leq j,\ell\leq 2}.
\]
Denote $|g(X)|=\det(g(X))$. Then the magnetic field $\bf B$ is given
by
\[
{\bf B}=B\,dx\wedge dy, \quad B =\frac{\partial A_2}{\partial
x}-\frac{\partial A_1}{\partial y},
\]
and
\[
B=b\sqrt{|g|}.
\]
Finally, the operator $H^h$ has the form
\[
H^h=\frac{1}{\sqrt{|g(X)|}}\sum_{1\leq \alpha,\beta\leq
2}\nabla^h_\alpha \left(\sqrt{|g(X)|} g^{\alpha\beta}(X)
\nabla^h_\beta\right),
\]
where
\[
\nabla^h_\alpha= i h \frac{\partial}{\partial X^\alpha}+A_\alpha(X),
\quad \alpha =1,2,
\]
or, equivalently,
\[
H^h=\sum_{1\leq \alpha,\beta\leq 2} g^{\alpha\beta}(X)
\nabla^h_\alpha\nabla^h_\beta+i h \sum_{1\leq \alpha\leq
2}\Gamma^\alpha \nabla^h_\alpha,
\]
where, for $\alpha =1,2$,
\begin{equation}\label{e:Gamma}
\Gamma^\alpha=\frac{1}{\sqrt{|g(X)|}}\sum_{1\leq \beta\leq
2}\frac{\partial}{\partial X^\beta} \left(\sqrt{|g(X)|}
g^{\beta\alpha}(X)\right).
\end{equation}

We will consider normal local coordinates near $x_0$ such that $x_0$
corresponds to $(0,0)$ and, in a neighborhood of $x_0$,
\[
b(X)=b_0+\alpha_1 x^2+\beta_1 y^2+O(|X|^3).
\]
Thus, we have
\[
a=(\alpha_1)^{1/2}+(\beta_1)^{1/2}, \quad d=\alpha_1\beta_1, \quad
t=\alpha_1+\beta_1.
\]

By well-known properties of normal coordinates we have
\[
g_{11}(X)=1+O(|X|^2),\quad g_{12}(X)=O(|X|^2),\quad
g_{22}(X)=1+O(|X|^2).
\]
Moreover (see, for instance, \cite[Proposition 1.28]{BGV92}), we
have
\[
g_{ij}(X)=\delta_{ij}-\frac{1}{3}\sum_{kl}R_{ikjl}(x_0)X^kX^l+O(|X|^3),
\]
where $R_{ijkl}$ is the Riemann curvature tensor. Therefore,
Taylor's expansion of $g^{\alpha\beta}$ has the form
\begin{equation}\label{e:expg}
g^{\alpha\beta}(X)=\delta^{\alpha\beta}+\sum_{k=2}^\infty
g^{\alpha\beta}_{(k)}(X),
\end{equation}
where
\[
g^{\alpha\beta}_{(2)}(X)=\frac{1}{3}\sum_{kl}R_{\alpha k\beta
l}(x_0) X^kX^l.
\]
In the two-dimensional case, due to its symmetries, the Riemann
curvature tensor is determined by a single component
\[
R_{1212}=-R_{2112}=R_{2121}=-R_{1221}.
\]
The other components equal zero. We have
\[
2R_{1212}=R(g_{11}g_{22}-g^2_{12}),
\]
where $R$ is the scalar curvature. So we have
\[
R_{1212}(x_0)=\frac{1}{2}R(x_0).
\]

Thus we have
\begin{equation}\label{e:g2}
g^{11}_{(2)}(X)=\frac{1}{6} R(x_0) y^2, \quad g^{1
2}_{(2)}(X)=-\frac{1}{6} R(x_0) x y, \quad
g^{22}_{(2)}(X)=\frac{1}{6} R(x_0) x^2.
\end{equation}

We also have
\begin{equation}\label{e:detg}
\sqrt{|g(X)|}=1-\frac{1}{12} R(x_0) x^2-\frac{1}{12} R(x_0)
y^2+O(|X|^3).
\end{equation}

Let us write Taylor's expansion of $\Gamma^\alpha$ in the form
\begin{equation}\label{e:expgamma}
\Gamma^\alpha(X)=\sum_{k=0}^\infty\Gamma^{\alpha}_{(k)}(X), \quad
\alpha=1,2.
\end{equation}
Using \eqref{e:g2} and \eqref{e:detg}, one can show that
\[
\Gamma^{\alpha}_{(0)}(X)=0, \quad \alpha=1,2,
\]
and
\begin{equation}\label{e:gamma1}
\Gamma^{1}_{(1)}(X)=-\frac{1}{3}R(x_0)x\,,\quad
\Gamma^{2}_{(1)}(X)=-\frac{1}{3}R(x_0)y\,.
\end{equation}

If we write ${\bf B}=B(x,y)\,dx\,dy$ then
\[
B(X)=b(X)\sqrt{|g(X)|}=b_0+\alpha x^2+\beta y^2+O(|X|^3),
\]
where
\[
\alpha_1=\alpha+\frac{1}{12}b_0 R(x_0)>0, \quad
\beta_1=\beta+\frac{1}{12}b_0 R(x_0)>0.
\]
We can also choose a gauge $A$ such that
\[
A_1(X)=0\ \text{and}\ A_2(X)=b_0x+\frac{\alpha}{3}x^3+\beta
xy^2+O(|X|^4).
\]
We expand $A_2$ into the Taylor series:
\[
A_2(X)=b_0x+\sum_{j=3}^\infty S_j(x,y),
\]
with
\[
S_j(x,y)=\sum_{\ell=0}^j S_{j\ell}x^\ell y^{j-\ell}.
\]
In particular, we have
\[
S_3(x,y)=\frac{\alpha}{3}x^3+\beta xy^2.
\]

Next we move the operator $H^h$ into the Hilbert space $L^2(\RR^n)$
equipped with the Euclidean inner product, considering the operator
\begin{align*}
\hat{H}^h& =|g(X)|^{1/4}H_h |g(X)|^{-1/4}\\
& =\sum_{1\leq \alpha,\beta\leq 2} g^{\alpha\beta}(X)
\hat{\nabla}^h_\alpha\hat{\nabla}^h_\beta+i h \sum_{1\leq \alpha\leq
2}\Gamma^\alpha \hat{\nabla}^h_\alpha,
\end{align*}
where, for $\alpha =1,2$,
\begin{equation}\label{e:nabla}
\hat{\nabla}^h_\alpha=|g(X)|^{1/4} \nabla^h_\alpha |g(X)|^{-1/4} =
\nabla^h_\alpha + \frac{1}{12}i h R(x_0) X^\alpha+O(h|X|^2)).
\end{equation}

Now we use the scaling $x=h^{1/2}x_1, y=h^{1/2}y_1$ and expand the
resulting operator $\hat{H}^h(x_1,y_1,D_{x_1},D_{y_1})$ into a
formal power series of $h^{1/2}$. By \eqref{e:nabla}, we have
\begin{align*}
\hat{\nabla}^h_1& = h^{1/2}(-D_{x_1} + \frac{1}{12}i h
R(x_0)x_1+O(h^{3/2})),\\
\hat{\nabla}^h_2 &= h^{1/2}(-D_{y_1} +b_0x_1+ \frac{1}{12}i h
R(x_0)y_1+hS_3(x_1,y_1)+O(h^{3/2})).
\end{align*}
From \eqref{e:expg} and \eqref{e:expgamma}, we get
\begin{align*}
g^{\alpha\beta}(x_1,y_1)&=\delta^{\alpha\beta}+hg^{\alpha\beta}_{(2)}(x_1,y_1)
+\sum_{k=3}^\infty h^{k/2}g^{\alpha\beta}_{(k)}(x_1,y_1),\\
\Gamma^\alpha(x_1,y_1)&=h^{1/2}
\Gamma^{\alpha}_{(1)}(x_1,y_1)+\sum_{k=2}^\infty h^{k/2}
\Gamma^{\alpha}_{(k)}(x_1,y_1), \quad \alpha=1,2.
\end{align*}
Using these expansions, one can check that the operator $\hat{H}^h$
has the form
\[
\hat{H}^h(x_1,y_1,D_{x_1},D_{y_1})=h Q^h(x_1,y_1,D_{x_1},D_{y_1}),
\]
with
\[
Q^h(x_1,y_1,D_{x_1},D_{y_1})=\sum_{k=0}^\infty
h^{k/2}Q_k(x_1,y_1,D_{x_1},D_{y_1}),
\]
where
\begin{align*}
Q_0(x_1,y_1,D_{x_1},D_{y_1})&=D_{x_1}^2 +(D_{y_1}-b_0x_1)^2,\\
Q_1(x_1,y_1,D_{x_1},D_{y_1})&=0,
\end{align*}
and
\begin{align*}
Q_2(x_1,y_1,D_{x_1},D_{y_1}) = & -\frac{1}{12}iR(x_0)
(x_1D_{x_1}+D_{x_1}x_1)\\
& -\frac{1}{12}iR(x_0) (y_1(D_{y_1}-b_0x_1)+(D_{y_1}-b_0x_1)y_1)
\\ & -((D_{y_1}-b_0x_1)S_3(x_1,y_1)+S_3(x_1,y_1)(D_{y_1}-b_0x_1))\\
& + g^{11}_{(2)}(x_1,y_1)
D^2_{x_1}\\
& + g^{12}_{(2)}(x_1,y_1)\,  D_{x_1}\, (D_{y_1}-b_0x_1)
\\ & + g^{12}_{(2)}(x_1,y_1) \,
(D_{y_1}-b_0x_1)\, D_{x_1}\\
& + g^{22}_{(2)}(x_1,y_1)
(D_{y_1}-b_0x_1)^2\\
& -i \Gamma^{1}_{(1)}(x_1,y_1)D_{x_1}\\
& -i \Gamma^{2}_{(1)}(x_1,y_1) (D_{y_1}-b_0x_1).
\end{align*}
Using the Fourier transform in $y_1$, we can write the operator
$Q^h$ as
\[
Q^h(x_1,-D_\xi,D_{x_1},\xi)=\sum_{k=0}^\infty
h^{k/2}Q_k(x_1,-D_\xi,D_{x_1},\xi),
\]
where
\begin{align*}
Q_0(x_1,-D_\xi,D_{x_1},\xi)&=D_{x_1}^2 +(\xi-b_0x_1)^2,\\
Q_1(x_1,-D_\xi,D_{x_1},\xi)&=0,
\end{align*}
and
\begin{align*}
Q_2(x_1,-D_\xi,D_{x_1},\xi) = & -\frac{1}{12}iR(x_0)
(2x_1D_{x_1}-i)\\
& +\frac{1}{12}iR(x_0) (2(\xi-b_0x_1) D_{\xi}-i)
\\ & -((\xi-b_0x_1) S_3(x_1,-D_\xi)+S_3(x_1,-D_\xi)(\xi-b_0x_1))\\
& + g^{11}_{(2)}(x_1,-D_\xi)
D^2_{x_1}\\
& + g^{12}_{(2)}(x_1,-D_\xi)\, D_{x_1}\, (\xi-b_0x_1)
\\ & + g^{12}_{(2)}(x_1,-D_\xi) \,
(\xi-b_0x_1)\, D_{x_1}\\
& + g^{22}_{(2)}(x_1,-D_\xi)
(\xi-b_0x_1)^2\\
& -i \Gamma^{1}_{(1)}(x_1,-D_\xi)D_{x_1}\\
& -i \Gamma^{2}_{(1)}(x_1,-D_\xi) (\xi-b_0x_1).
\end{align*}
A further translation $x_2=x_1-\frac{\xi}{b_0}$ gives
\[
\hat{H}^h=h T^h(x_2,\xi,D_{x_2},D_{\xi}),
\]
where
\[
T^h(x_2,\xi,D_{x_2},D_{\xi})=Q^h\left(x_2+\frac{\xi}{b_0},-D_\xi+\frac{1}{b_0}D_{x_2},D_{x_2},\xi\right).
\]
We have (denoting $w=(x_2,\xi)$)
\[
T^h(w,D_{w})=T_0(x_2,D_{x_2})+h T_2(w,D_{w})+\sum_{j=3}^\infty
h^{j/2}T_j(w,D_{w})
\]
with
\[
T_0(x_2,D_{x_2})=D^2_{x_2}+ b_0^2x_2^2,
\]
and
\begin{align*}
T_2(w,D_{w}) = & -\frac{1}{12}iR(x_0)
(2\left(x_2+\frac{\xi}{b_0}\right)D_{x_2}-i)\\
& +\frac{1}{12}iR(x_0) (2b_0x_2\left(
-D_\xi+\frac{1}{b_0}D_{x_2}\right)+i)
\\ & +(b_0x_2 \hat{S}_3(x_2,\xi,D_{x_2},D_\xi)+\hat{S}_3(x_2,\xi,D_{x_2},D_\xi)b_0x_2)\\
& + g^{11}_{(2)}(x_2+\frac{\xi}{b_0},-D_\xi+\frac{1}{b_0}D_{x_2})
D^2_{x_2}\\
& - g^{12}_{(2)}(x_2+\frac{\xi}{b_0},-D_\xi+\frac{1}{b_0}D_{x_2})\,
D_{x_2}\,(b_0x_2)
\\ & - g^{12}_{(2)}(x_2+\frac{\xi}{b_0},-D_\xi+\frac{1}{b_0}D_{x_2}) \,
b_0x_2\, D_{x_2}\\
& + g^{22}_{(2)}(x_2+\frac{\xi}{b_0},-D_\xi+\frac{1}{b_0}D_{x_2})
b_0^2x_2^2\\
& -i \Gamma^{1}_{(1)}(x_2+\frac{\xi}{b_0},-D_\xi+\frac{1}{b_0}D_{x_2})D_{x_2}\\
& +i
\Gamma^{2}_{(1)}(x_2+\frac{\xi}{b_0},-D_\xi+\frac{1}{b_0}D_{x_2})
b_0x_2,
\end{align*}
where
\[
\hat{S}_3(x_2,\xi,D_{x_2},D_\xi)=S_3\left(x_2+\frac{\xi}{b_0},-D_\xi+\frac{1}{b_0}D_{x_2}\right).
\]
The operator $\hat{S}_3$ has the following form:
\begin{multline*}
\hat{S}_3(x_2,\xi,D_{x_2},D_\xi)=x_2L(\xi,D_\xi)+M_0(\xi,D_\xi)+M_1(x_2,D_{x_2})
\\ +M_2(x_2,D_{x_2},D_\xi)+
M_3(x_2,D_{x_2},D_\xi)+M_4(x_2,\xi,D_{x_2}),
\end{multline*}
where
\begin{align*}
L(\xi,D_\xi) & = \frac{\alpha}{b^2_0}\xi^2+\beta D^2_\xi, \\
M_0(\xi,D_\xi) & = \frac{\alpha}{3b^3_0}\xi^3+\frac{\beta}{b_0}\xi D^2_\xi, \\
M_1(x_2,D_{x_2}) & = \frac{\alpha}{3}x_2^3+\frac{\beta}{b_0^2}x_2D_{x_2}^2, \\
M_2(x_2,D_{x_2},D_\xi) & = -2\frac{\beta}{b_0^2}D_{x_2}\xi D_\xi, \\
M_3(x_2,D_{x_2},D_\xi) & = -2\frac{\beta}{b_0}x_2 D_{x_2} D_\xi, \\
M_4(x_2,\xi,D_{x_2})& =
\frac{\alpha}{b_0}x^2_2\xi+\frac{\beta}{b_0^3}\xi D_{x_2}^2.
\end{align*}
So we get
\begin{align*}
T_2(w,D_{w})= & -\frac{1}{6}i R(x_0)\frac{\xi}{b_0}D_{x_2}-\frac{1}{6}i R(x_0)b_0x_2 D_\xi \\
& + 2b_0x_2^2L(\xi,D_\xi)+ b_0 \left(x_2 M(w,D_w)+M(w,D_w)x_2)\right)\\
& +
g^{11}_{(2)}\left(x_2+\frac{\xi}{b_0},-D_\xi+\frac{1}{b_0}D_{x_2}\right)
D^2_{x_2}\\
& -
g^{12}_{(2)}\left(x_2+\frac{\xi}{b_0},-D_\xi+\frac{1}{b_0}D_{x_2}\right)
\Big( D_{x_2} (b_0x_2) + b_0x_2D_{x_2}\Big)\\
& +
g^{22}_{(2)}\left(x_2+\frac{\xi}{b_0},-D_\xi+\frac{1}{b_0}D_{x_2}\right)
b^2_0x^2_2\\
& -i \Gamma^{1}_{(1)}\left(x_2+\frac{\xi}{b_0},-D_\xi+\frac{1}{b_0}D_{x_2}\right)D_{x_2}\\
& +i \Gamma^{2}_{(1)}
\left(x_2+\frac{\xi}{b_0},-D_\xi+\frac{1}{b_0}D_{x_2}\right) b_0x_2,
\end{align*}
where
\[
M(w,D_w)=\sum_{\ell=0}^3M_\ell(w,D_w).
\]

We have
\[
{\rm Sp}(T_0(x_2,D_{x_2}))=\{\mu_{k}=(2k+1)b_0: k\in \NN\}.
\]
The eigenfunction of $T_0(x_2,D_{x_2})$ associated to the eigenvalue
$\mu_{k}$ is
\[
\psi_{k}(x_2)=\pi^{-1/4}b_0^{1/2} H_{k}(b_0^{1/2}x_2)
e^{-b_0x^2_2/2},
\]
where $H_k$ is the Hermite polynomial:
\[
H_k(x)=(-1)^ke^{x^2}\frac{d^k}{dx^k}(e^{-x^2}).
\]
The norm of $\psi_{k}$ in $L^2(\RR,dx)$ equals the norm of $H_k$ in
$L^2(\RR,e^{-x^2}\,dx)$, which is given by
\[
\|H_k\|=\sqrt{2^kk!\sqrt{k}}.
\]

\subsection{Construction of approximate eigenfunctions}
First, we construct a formal eigenfunction $u^h$ of the operator
$T^h(w,D_{w})$ admitting an asymptotic expansion in the form of a
formal asymptotic series in powers of $h^{1/2}$
\[
u^h=\sum_{\ell=0}^\infty u^{(\ell)}h^{\ell/2}, \quad u^{(\ell)}\in
{\cS(\RR^2)},
\]
with the corresponding formal eigenvalue
\[
\lambda^h=\sum_{\ell=0}^\infty\lambda^{(\ell)}h^{\ell/2},
\]
such that
\[
T^h(w,D_{w})u^h-\lambda^h u^h=0
\]
in the sense of asymptotic series in powers of $h^{1/2}$.
\medskip\par
{\bf The first terms.} Looking at the coefficient of $h^0$,  we
obtain:
\[
T_0(x_2,D_{x_2})u^{(0)}=\lambda^{(0)}u^{(0)}\,.
\]
Thus, we have
\begin{equation}\label{e:0}
\lambda^{(0)}=\lambda^{(0)}_k=(2k+1)b_0, \quad
u^{(0)}(x_2,\xi)=\frac{1}{\|H_k\|}\psi_{k}(x_2)\chi_0(\xi), \quad
k\in\NN,
\end{equation}
where $\chi_0$ is some function, which will be determined later, and
$c$ is some constant.

Looking at the coefficient of $h^{1/2}$,  we obtain:
\[
T_0(x_2,D_{x_2})u^{(1)}=\lambda^{(0)}u^{(1)}+\lambda^{(1)}u^{(0)}.
\]
The orthogonality condition implies that
\[
\lambda^{(1)}=0.
\]
Under this condition, we get
\begin{equation}\label{e:1/2}
u^{(1)}(x_2,\xi)=\frac{1}{\|H_k\|}\psi_{k}(x_2)\chi_1(\xi),
\end{equation}
where $\chi_1$ is some function, which will be determined later.

Next, the cancelation of the coefficient of $h^1$ gives:
\begin{equation}\label{e:1}
( T_0(x_2,D_{x_2})-\lambda^{(0)})u^{(2)}= \lambda^{(2)}u^{(0)} -
T_2(w,D_{w})u^{(0)}.
\end{equation}
The orthogonality condition implies that
\[
\lambda^{(2)}\chi_0(\xi)-\frac{1}{\|H_k\|}\int T_2(w,D_{w})u^{(0)}
\psi_{k}(x_2) dx_2=0.
\]

\begin{lem}
For any function $u$ of the form
$u(x_2,\xi)=\frac{1}{\|H_k\|}\psi_{k}(x_2)\chi(\xi)$, we have
\[
\frac{1}{\|H_k\|}\int T_2(w,D_{w})u(w) \psi_{k}(x_2) dx_2 =
\cH_k\chi(\xi),
\]
where $\cH_k$ is the harmonic oscillator:
\begin{align*}
\cH_k=& (2k+1)\beta_1D^2_\xi + (2k+1)\alpha_1\frac{1}{b^2_0} \xi^2\\
& +\frac{1}{2b_0} (2k^2+2k+1)\left(\alpha_1+\beta_1+\frac{1}{2} b_0
R(x_0)\right)-\frac{1}{4}R(x_0).
\end{align*}
\end{lem}

\begin{proof}
We have
\[
D_{x_2}\psi_{k}=-ib_0^{1/2}\left(2k\psi_{k-1}- b_0^{1/2}
x_2\psi_{k}\right),
\]
\[
D^2_{x_2}\psi_{k}=b_0(2k+1-b_0x_2^2)\psi_{k},
\]
and
\[
D^3_{x_2}\psi_{k}= 2ib_0^2x_2\psi_{k} +
b_0(2k+1-b_0x_2^2)D_{x_2}\psi_{k}.
\]
We also have
\begin{align*}
2xH_k= & H_{k+1}+2kH_{k-1},\\
4x^2H_k= & H_{k+2}+(4k+2)H_k+4k(k-1)H_{k-2},\\
8x^3H_k= & H_{k+3}+(6k+6)H_{k+1}+12k^2H_{k-1}+8k(k-1)(k-2)H_{k-3}\\
16x^4H_k(x)= & H_{k+4}(x)+(8k+12)H_{k+2}(x)+12(2k^2+2k+1)H_k(x)\\
&+16(2k^2-3k+1)kH_{k-2}(x)+16k(k-1)(k-2)(k-3)H_{k-4}(x),
\end{align*}
that implies that
\begin{align*}
\frac{1}{\|H_k\|^2}\langle x_2\psi_{k-1}, \psi_{k}\rangle = & \frac{1}{2b_0^{1/2}},\\
\frac{1}{\|H_k\|^2}\langle x^3_2\psi_{k-1}, \psi_{k}\rangle = & \frac{3}{4b_0^{3/2}}k,\\
\frac{1}{\|H_k\|^2}\langle x^2_2\psi_{k}, \psi_{k}\rangle = & \frac{1}{2b_0}(2k+1),\\
\frac{1}{\|H_k\|^2}\langle x^4_2 \psi_{k}, \psi_{k}\rangle = &
\frac{3}{4b_0^2}(2k^2+2k+1).
\end{align*}
Next, we have
\begin{align*}
\frac{1}{\|H_k\|^2}\langle x_2D_{x_2} \psi_{k}, \psi_{k}\rangle = &
\frac{-ib_0^{1/2}}{\|H_k\|^2}\langle 2kx_2\psi_{k-1}-
b_0^{1/2}x^2_2\psi_{k}, \psi_{k}\rangle\\
= & \frac{1}{2}i,\\
\frac{1}{\|H_k\|^2} \langle D^2_{x_2} \psi_{k}, \psi_{k}\rangle = &
\frac{b_0}{\|H_k\|^2}
\langle (2k+1-b_0x_2^2)\psi_{k}, \psi_{k}\rangle \\
= & \frac{1}{2}(2k+1)b_0,\\
\frac{1}{\|H_k\|^2}\langle x_2^2 D^2_{x_2} \psi_{k}, \psi_{k}\rangle
= & \frac{b_0}{\|H_k\|^2}
\langle ((2k+1)x_2^2-b_0x_4^2)\psi_{k}, \psi_{k}\rangle \\
= & \frac{1}{4}(2k^2+2k-1),\\
 \frac{1}{\|H_k\|^2}\langle D^4_{x_2}
\psi_{k}, \psi_{k}\rangle = & \frac{b_0^2}{\|H_k\|^2}
\langle (2k+1-b_0x_2^2)^2\psi_{k}, \psi_{k}\rangle \\
= & \frac{3}{4}(2k^2+2k+1)b_0^2.
\end{align*}

First, remark that
\[
\delta_0\lambda= \frac{1}{\|H_k\|}\int \left(-\frac{1}{6}i R(x_0)
\frac{\xi}{b_0}D_{x_2}-\frac{1}{6}i R(x_0)b_0x_2 D_\xi\right)
u(w)\psi_k(x_2)dx_2=0,
\]
since the operator $-\frac{1}{6}i R(x_0)
\frac{\xi}{b_0}D_{x_2}-\frac{1}{6}i R(x_0)b_0x_2 D_\xi$ is odd in
the $x_2$ variable.

Next, we have
\begin{align*}
\delta_1\lambda= & \frac{2b_0}{\|H_k\|^2} \langle x_2^2\psi_{k},
\psi_{k}\rangle
L(\xi,D_\xi)\chi(\xi)\\
= & (2k+1)\left(\frac{\alpha}{b^2_0}\xi^2+\beta
D^2_\xi\right)\chi(\xi).
\end{align*}

Now we consider
\[
\delta_2\lambda= \frac{1}{\|H_k\|}\int b_0 \left(x_2M(w,D_w)
+M(w,D_w) x_2\right)u(w)\psi_k(x_2)dx_2.
\]
The operators $M_1(w,D_w)$ and $M_2(w,D_w)$ are odd in the $x_2$
variable, so we obtain
\begin{align*}
\delta_2\lambda = & \frac{1}{\|H_k\|}\int b_0 \left(x_2M_1(w,D_w)
+M_1(w,D_w)
x_2\right)u(w)\psi_k(x_2)dx_2\\
& + \frac{1}{\|H_k\|}\int b_0 \left(x_2M_2(w,D_w) +M_2(w,D_w)
x_2\right)u(w)\psi_k(x_2)dx_2\\
= & \frac{b_0}{\|H_k\|^2} \langle  \left(
\frac{2\alpha}{3}x_2^4+\frac{\beta}{b_0^2}x^2_2D_{x_2}^2
+\frac{\beta}{b_0^2}x_2D_{x_2}^2
x_2\right)\psi_{k}, \psi_{k}\rangle \chi(\xi)\\
& - \frac{b_0}{\|H_k\|^2} \langle \left(2\frac{\beta}{b_0^2}x_2
D_{x_2} + 2\frac{\beta}{b_0^2}D_{x_2}x_2\right)\psi_{k},
\psi_{k}\rangle \xi D_\xi \chi(\xi) \\
= & \frac{2b_0}{\|H_k\|^2} \langle  \left(
\frac{\alpha}{3}x_2^4+\frac{\beta}{b_0^2}x^2_2D_{x_2}^2
-i\frac{\beta}{b_0^2}x_2D_{x_2}\right) \psi_{k}, \psi_{k}\rangle  \chi(\xi) \\
& + \frac{2b_0}{\|H_k\|^2} \langle \left(-\frac{2\beta}{b_0^2}x_2
D_{x_2} +i \frac{\beta}{b_0^2}\right)\psi_{k}, \psi_{k}\rangle \xi
D_\xi \chi(\xi).
\end{align*}
Thus we arrive at
\[
\delta_2\lambda =\frac{\alpha+\beta}{2b_0} (2k^2+2k+1) \chi(\xi).
\]

Next, we consider
\[
\delta_3\lambda = \frac{1}{\|H_k\|}\int g^{11}_{(2)}
\left(x_2+\frac{\xi}{b_0},-D_\xi+\frac{1}{b_0}D_{x_2}\right)
D^2_{x_2}u(w)\psi_k(x_2)dx_2.
\]
By \eqref{e:g2}, we have
\[
g^{11}_{(2)}
\left(x_2+\frac{\xi}{b_0},-D_\xi+\frac{1}{b_0}D_{x_2}\right)=
\frac{1}{6} R(x_0)(-D_\xi+\frac{1}{b_0}D_{x_2})^2.
\]
Therefore,
\[
\delta_3\lambda = \frac{1}{6\|H_k\|}R(x_0) \int
(-D_\xi+\frac{1}{b_0} D_{x_2})^2D^2_{x_2}u(w)\psi_k(x_2)dx_2.
\]
It suffices to consider the terms, which are even with respect to
$x_2$:
\[
\delta_3\lambda = \frac{1}{6\|H_k\|}R(x_0) \int
(D^2_\xi+\frac{1}{b^2_0}D^2_{x_2})D^2_{x_2}u(w)\psi_k(x_2)dx_2.
\]
Thus, we obtain that
\begin{align*}
\delta_3\lambda = & \frac{1}{6\|H_k\|^2} R(x_0) \langle
D^2_{x_2}\psi_k, \psi_k\rangle D^2_\xi\chi(\xi)\\
& + \frac{1}{6b^2_0\|H_k\|^2} R(x_0) \langle D^4_{x_2}\psi_k,
\psi_k\rangle \chi(\xi)\\
= & \frac{1}{12}(2k+1)b_0 R(x_0)D^2_\xi\chi(\xi)\\
& + \frac{1}{8}(2k^2+2k+1)R(x_0)\chi(\xi).
\end{align*}

Next, we consider
\[
\delta_4\lambda = -\frac{1}{\|H_k\|}\int g^{12}_{(2)}
\left(x_2+\frac{\xi}{b_0},-D_\xi+\frac{1}{b_0}D_{x_2}\right) \Big(
D_{x_2} (b_0x_2) + b_0x_2D_{x_2}\Big) u(w)\psi_k(x_2)dx_2.
\]
By \eqref{e:g2}, we have
\[
\delta_4\lambda = \frac{1}{6\|H_k\|} R(x_0) \int
\left(x_2+\frac{\xi}{b_0}\right)
\left(-D_\xi+\frac{1}{b_0}D_{x_2}\right)\left(
2b_0x_2D_{x_2}-ib_0\right)u(w)\psi_k(x_2)dx_2.
\]
It suffices to consider the terms, which are even with respect to
$x_2$:
\begin{align*}
\delta_4\lambda = & \frac{1}{6\|H_k\|}R(x_0) \int (x_2D_{x_2}-\xi
D_\xi)\Big(2x_2D_{x_2}-i\Big)u(w)\psi_k(x_2)dx_2 \\
= & \frac{1}{3\|H_k\|^2 }R(x_0) \langle x^2_2D^2_{x_2}\psi_k,
\psi_k\rangle \chi(\xi) \\
& - \frac{1}{2\|H_k\|^2 }iR(x_0) \langle x_2D_{x_2} \psi_k,
\psi_k\rangle \chi(\xi)\\
& - \frac{1}{6\|H_k\|^2 }R(x_0) \langle (2x_2D_{x_2}-i)\psi_k,
\psi_k \rangle \xi D_\xi \chi(\xi)
 \\
= & \left(\frac{1}{12}(2k^2+2k+1)R(x_0)+\frac{1}{12}R(x_0)\right)
\chi(\xi).
\end{align*}

Consider
\[
\delta_5\lambda = \frac{1}{\|H_k\|} \int g^{22}_{(2)}
\left(x_2+\frac{\xi}{b_0},-D_\xi+\frac{1}{b_0}D_{x_2}\right)
b_0^2x_2^2u(w)\psi_k(x_2)dx_2.
\]
By \eqref{e:g2}, we have
\[
\delta_5\lambda = \frac{1}{6\|H_k\|} R(x_0) \int
(x_2+\frac{\xi}{b_0})^2b_0^2x_2^2u(w)\psi_k(x_2)dx_2.
\]
It suffices to consider the terms, which are even with respect to
$x_2$:
\begin{align*}
\delta_5\lambda = & \frac{1}{6\|H_k\|} R(x_0) \int
(x^2_2+\frac{\xi^2}{b^2_0})b_0^2
x_2^2u(w)\psi_k(x_2)dx_2 \\
= & \frac{1}{6\|H_k\|^2 }b_0^2 R(x_0) \langle x^4_2\psi_{k},
\psi_{k}\rangle \chi(\xi) + \frac{1}{6\|H_k\|^2 } R(x_0) \langle
x_2^2\psi_{k}, \psi_{k}\rangle \xi^2 \chi(\xi) \\
= & \frac{1}{8} (2k^2+2k+1) R(x_0)\chi(\xi)+ \frac{1}{12b_0}(2k+1)
R(x_0)\xi^2 \chi(\xi).
\end{align*}

Next, consider
\[
\delta_6\lambda = -\frac{i}{\|H_k\|} \int
\Gamma^{1}_{(1)}\left(x_2+\frac{\xi}{b_0},-D_\xi+\frac{1}{b_0}D_{x_2}\right)D_{x_2}
u(w)\psi_k(x_2)dx_2.
\]
By \eqref{e:gamma1}, we have
\[
\delta_6\lambda = \frac{i}{3\|H_k\|}R(x_0) \int
\left(x_2+\frac{\xi}{b_0}\right)D_{x_2} u(w)\psi_k(x_2)dx_2.
\]
It suffices to consider the terms, which are even with respect to
$x_2$:
\begin{align*}
\delta_6\lambda = & \frac{i}{3\|H_k\|^2 }R(x_0) \langle x_2
D_{x_2} \psi_{k}, \psi_{k}\rangle \chi(\xi) \\
= & - \frac{1}{6}R(x_0)\chi(\xi).
\end{align*}

Finally, we take
\[
\delta_7\lambda = \frac{i}{\|H_k\|} \int
\Gamma^{2}_{(1)}\left(x_2+\frac{\xi}{b_0},-D_\xi+\frac{1}{b_0}D_{x_2}\right)b_0x_2
u(w)\psi_k(x_2)dx_2.
\]
By \eqref{e:gamma1}, we have
\[
\delta_7\lambda = \frac{i}{3\|H_k\|}R(x_0) \int
\left(-D_\xi+\frac{1}{b_0}D_{x_2}\right)b_0x_2 u(w)\psi_k(x_2)dx_2.
\]
It suffices to consider the terms, which are even with respect to
$x_2$:
\begin{align*}
\delta_7\lambda = & \frac{i}{3\|H_k\|^2}R(x_0) \langle
(x_2D_{x_2}-i)
\psi_{k}(w), \psi_{k}(w)\rangle \chi(\xi)\\
= & - \frac{1}{6}R(x_0)\chi(\xi).
\end{align*}

We conclude
\begin{align*}
\lefteqn{\frac{1}{\|H_k\|} \int T_2(w,D_{w})u(w) \psi_{k}(x_2) dx_2} \\
= & \delta_0\lambda+
\delta_1\lambda+\delta_2\lambda+\delta_3\lambda+
\delta_4\lambda+\delta_5\lambda+\delta_6\lambda+\delta_7\lambda \\
= & \frac{1}{2b_0} (2k^2+2k+1)\left(\alpha+\beta+\frac{2}{3} b_0
R(x_0)\right) \chi(\xi) \\ & +(2k+1)\left(\beta +
\frac{1}{12} b_0 R(x_0)\right)D^2_\xi\chi(\xi)\\
& + (2k+1) \left(\alpha+ \frac{1}{12} b_0 R(x_0) \right)
\frac{1}{b^2_0}\xi^2 \chi(\xi)
\\ & -\frac{1}{4}R(x_0)\chi(\xi)\\
= & \frac{1}{2b_0} (2k^2+2k+1)\left(\alpha_1+\beta_1+\frac{1}{2} b_0
R(x_0)\right)\chi(\xi) \\ & +(2k+1)\beta_1D^2_\xi\chi(\xi)+
(2k+1)\alpha_1 \frac{1}{b^2_0}\xi^2 \chi(\xi)
-\frac{1}{4}R(x_0)\chi(\xi).
\end{align*}
\end{proof}

Thus, we obtain that
\[
\lambda^{(2)}=\lambda^{(2)}_{jk}=\nu_{jk},\quad j,k\in \NN,
\]
where $\nu_{jk}$ is an eigenvalue of the harmonic oscillator
$\cH_k$:
\begin{multline*}
\nu_{jk}=(2j+1)(2k+1)(\alpha_1\beta_1)^{1/2}\frac{1}{b_0}\\
+\frac{1}{2b_0}
(2k^2+2k+1)\left(\alpha_1+\beta_1\right)+\frac{1}{2}(k^2+k) R(x_0),
\quad j,k\in\NN,
\end{multline*}
and
\[
\chi_0(\xi)=\Psi_{jk}(\xi),
\]
where $\Psi_{jk}$ is the normalized eigenfunction of $\cH_k$
associated to the eigenvalue $\nu_{jk}$.

Moreover, we conclude that $u^{(2)}$ is a solution of \eqref{e:1},
which can be written as
\[
u^{(2)}=\phi^{(2)}(x_2,\xi)+\psi_{k}(x_2)\chi_2(\xi),
\]
where $\phi^{(2)}$ is a solution of \eqref{e:1}, satisfying the
condition
\[
\int \phi^{(2)}(x_2,\xi)\psi_{k}(x_2)\,dx_2=0,
\]
and $\chi_2$ will be determined later.

Now the cancelation of the coefficient of $h^{3/2}$ gives:
\begin{multline}\label{e:3/2}
( T_0(x_2,D_{x_2})-\lambda^{(0)})u^{(3)}\\ = \lambda^{(3)}u^{(0)} -
T_3(w,D_{w})u^{(0)}+\lambda^{(2)}u^{(1)} - T_2(w,D_{w})u^{(1)}.
\end{multline}
The orthogonality condition for \eqref{e:3/2} is written as
\begin{equation}\label{e:3/2ort}
\lambda^{(3)}\chi_0 - \frac{1}{\|H_k\|}\int
T_3(w,D_{w})u^{(0)}\psi_{k}(x_2)\,dx_2 +\lambda^{(2)}\chi_1 -
\cH_k\chi_1=0.
\end{equation}
Under this assumption, we obtain that $u^{(3)}$ is a solution of
\eqref{e:3/2}, which can be written as
\[
u^{(3)}=\phi^{(3)}(x_2,\xi)+\frac{1}{\|H_k\|}\psi_{k}(x_2)\chi_3(\xi),
\]
where $\phi^{(3)}$ is a solution of \eqref{e:3/2}, satisfying the
condition
\[
\int \phi^{(3)}(x_2,\xi)\psi_{k}(x_2)\,dx_2=0,
\]
and $\chi_3$ will be determined later.

The equation \eqref{e:3/2ort} has a solution if and only if
\[
\lambda^{(3)} = \frac{1}{\|H_k\|}\iint
T_3(w,D_{w})u^{(0)}(w)\psi_{k}(x_2)\Psi_{jk}(\xi)\,dx_2\,d\xi,
\]
which allow us to find $\lambda^{(3)}$. Under this condition, there
exists a unique solution $\chi_1$ of \eqref{e:3/2ort}, orthogonal to
$\chi_0$

\medskip\par
{\bf The iteration procedure.} Suppose that the coefficients of
$h^{\ell/2}$ equal zero for $\ell=0,\ldots,n-1$, $n>3$. Then we know
the coefficients $\lambda^{(\ell)}$ for $\ell=0,\ldots,n-1$. We also
know that $u^{(\ell)}$ for $\ell=0,\ldots,n-1$ can be written as
\[
u^{(\ell)}=\phi^{(\ell)}(x_2,\xi)+\frac{1}{\|H_k\|}\psi_{k}(x_2)\chi_\ell(\xi),
\]
where $\phi^{(\ell)}, \ell=0,\ldots,n-1$, are some known functions
in $\mathcal S(\mathbb R^2)$, satisfying the condition
\[
\int \phi^{(\ell)}(x_2,\xi)\psi_{k}(x_2)\,dx_2=0,
\]
and $\chi_\ell \in \mathbb S(\mathbb R)$ are known for
$\ell=0,\ldots,n-3$, $\chi_\ell \perp \chi_0$.

The cancelation of the coefficient of $h^{n/2}$ gives:
\begin{multline}\label{e:n/2}
(T_0(x_2,D_{x_2})-\lambda^{(0)})u^{(n)} = \lambda^{(n)}u^{(0)} -
T_n(w,D_{w})u^{(0)}\\
 +\sum_{\ell=3}^{n-1}(\lambda^{(\ell)}u^{(n-\ell)} -
T_\ell(w,D_{w})u^{(n-\ell)})+\lambda^{(2)}u^{(n-2)} -
T_2(w,D_{w})u^{(n-2)}.
\end{multline}
The orthogonality condition for \eqref{e:n/2} is written as
\begin{multline}\label{e:n/2ort}
\lambda^{(n)}\chi_0 - \frac{1}{\|H_k\|} \int
T_n(w,D_{w})u^{(0)}\psi_{k}(x_2)\,dx_2\\
+\sum_{\ell=3}^{n-1}(\lambda^{(\ell)}\chi_{n-\ell} -
\frac{1}{\|H_k\|} \int
T_\ell(w,D_{w})u^{(n-\ell)}\psi_{k}(x_2)\,dx_2)
\\ - \frac{1}{\|H_k\|} \int
T_\ell(w,D_{w})\phi^{(n-\ell)}\psi_{k}(x_2)\,dx_2
+\lambda^{(2)}\chi_{n-2} -\cH_k\chi_{n-2}=0.
\end{multline}

Under this assumption, we obtain that $u^{(n)}$ is a solution of
\eqref{e:n/2}, which can be written as
\[
u^{(n)}=\phi^{(n)}(x_2,\xi)+\frac{1}{\|H_k\|}\psi_{k}(x_2)\chi_n(\xi),
\]
where $\phi^{(n)}$ is a solution of \eqref{e:n/2}, satisfying the
condition
\[
\int \phi^{(n)}(x_2,\xi)\psi_{k}(x_2)\,dx_2=0,
\]
and $\chi_n$ will be determined later.

The orthogonality condition for \eqref{e:n/2ort} allows us to find
$\lambda^{(n)}$. Under this condition, there exists a unique
solution $\chi_{n-2}$ of \eqref{e:n/2ort}, orthogonal to $\chi_0$.
\medskip
\par
Thus, for any $j\in\NN$ and $k\in\NN$, we have constructed an
approximate eigenfunction $u^h_{jk}$ of the operator $T^h(w,D_{w})$
admitting an asymptotic expansion in the form of a formal asymptotic
series in powers of $h^{1/2}$
\[
u^h_{jk}=\sum_{\ell=0}^\infty u^{(\ell)}_{jk}h^{\ell/2}, \quad
u^{(\ell)}_{jk}\in {\cS(\RR^2)},
\]
such that
\begin{equation}\label{e:u0}
u^{(0)}_{jk}(x_2,\xi)=\frac{1}{\|H_k\|}\psi_{k}(x_2)\Psi_{jk}(\xi).
\end{equation}
with the corresponding approximate eigenvalue
\[
\lambda^h_{jk}=\sum_{\ell=0}^\infty\lambda^{(\ell)}_{jk}h^{\ell/2}.
\]
For any $N\in \NN$, consider
\[
u^h_{jk(N)}=\sum_{\ell=0}^N u^{(\ell)}_{jk}h^{\ell/2}, \quad
\lambda^h_{jk(N)}=\sum_{\ell=0}^N\lambda^{(\ell)}_{jk}h^{\ell/2}.
\]
Then we have
\[
T^h(w,D_{w})u^h_{jk(N)}-\lambda^h_{jk(N)}u^h_{jk(N)}=O(h^{\frac{N+1}{2}}).
\]

The constructed functions $u^h_{jk(N)}$ have sufficient decay
properties. Therefore, by changing back to the original coordinates
and multiplying by a fixed cut-off function, we obtain the desired
functions $\phi^h_{jkN}$, which satisfy \eqref{e:Hh} with
$\mu_{j,k,\ell}=\lambda^{(\ell)}_{jk}$.

The system $\{u^{(0)}_{jk}\}$ is an orthonormal system. Since each
change of variables, which we use, is unitary, this implies the
condition~\eqref{e:orth}.

\section{Lower bounds}\label{s:lower} In this section, we will prove the lower bound in
Theorem~\ref{t:main}. First, we recall a general lower bound due to
Montgomery \cite{Mont}. Suppose that $U$ is a domain in $M$. Then,
for any $u\in C^\infty_c(U)$, the following estimate holds:
\begin{equation}\label{e:Mont}
\|(ih\,d+{\bf A})u\|^2_U\geq \left|\int_U b |u|^2dx_g\right|.
\end{equation}
This fact is an immediate consequence of a Weitzenb\"ock-Bochner
type identity.

From \eqref{e:Mont}, it follows that we can restrict our
considerations by any sufficiently small neighborhood $\Omega$ of
$x_0$. Denote by $H^h_{D}$ the Dirichlet realization of the operator
$H^h$ in $L^2(\Omega,dx_g)$.

The estimate \eqref{e:Mont} implies that
\[
\tau h H^h_{D} +(1-\tau h)hb\leq H^h_{D}, \quad 0<\tau<h^{-1}.
\]
Taking $\tau=h^{-1/2}$, we obtain
\[
h^{1/2} (H^h_{D}-hb +h^{1/2}b)\leq H^h_{D}, \quad 0<h<1.
\]
Consider the Dirichlet realization $P^h_D$ of the operator
$H^h-hb+h^{1/2}(b-b_0)$ in $L^2(\Omega,dx_g)$. Then we have
\begin{equation}\label{e:lambdaj}
hb_0+h^{1/2} \lambda_j(P^h_{D})\leq \lambda_j(H^h_{D}).
\end{equation}
Therefore, the desired lower bound for $\lambda_j(H^h_{D})$ is an
immediate consequence of the following theorem.

\begin{thm}\label{t:PAV}
For any $j\in \NN$, there exist $C_j>0$ and $h_j>0$ such that
\[
\lambda_j(P^h_D)\geq h^{3/2}\left[\frac{2d^{1/2}}{b_0}j
+\frac{a^2}{2b_0}\right] - C_jh^{15/8}, \quad h\in (0,h_j].
\]
\end{thm}

To prove Theorem~\ref{t:PAV}, we will follow the lines of the proof
of \cite[Theorem 7.4]{HM01}. First, we observe that the upper bound
in Theorem~\ref{t:main} and \eqref{e:lambdaj} imply an upper bound
for $\lambda_j(P^h_D)$:
\begin{equation}\label{e:upperPh}
\lambda_j(P^h_D)\leq
h^{3/2}\left[\frac{2d^{1/2}}{b_0}j+\frac{a^2}{2b_0}\right] +
C_jh^{2}.
\end{equation}

For any eigenvalue $\lambda_j(P^h_D)$, denote by $u^{(j)}_h$ an
associated eigenfunction. By a straightforward repetition of the
arguments of \cite{HM01}, we can easily show the following analogue
of Lemmas 7.10 and 7.11 in \cite{HM01}\footnote{There are a few
inaccuracies in \cite{HM01}, concerning Lemma 7.11. For the erratum,
see http://www.math.u-psud.fr/~helffer/erratum164II.pdf }.

\begin{lem}\label{l:weight}
For any $j\in\NN$ and any real $k\geq 0$, we have
\[
\| |X|^ku^{(j)}_h\|_{L^2(\Omega,dx_g)}\leq
C_{k,j}(h^{k/2}+h^{(3k+1)/8}) \| u^{(j)}_h\|_{L^2(\Omega,dx_g)}.
\]
For any $j\in\NN$, any $\alpha =1,2$ and any $k\geq 0$, we have
\[
\| |X|^k\nabla^h_\alpha u^{(j)}_h\|_{L^2(\Omega,dx_g)}\leq
C_{k,j}(h^{(k+1)/2}+h^{(3k+5)/8})\|u^{(j)}_h\|_{L^2(\Omega,dx_g)}.
\]
\end{lem}

Take normal local coordinates near $x_0$ such that $x_0$ corresponds
to $(0,0)$ and, in a neighborhood of $x_0$,
\[
b(X)=b_0+\alpha_1 x^2+\beta_1 y^2+O(|X|^3).
\]
So we have
\[
g_{11}(X)=1+O(|X|^2),\quad g_{12}(X)=O(|X|^2),\quad
g_{22}(X)=1+O(|X|^2).
\]
We can take a magnetic potential $A$ such that
\[
A_1(X)=-\frac12b_0y+O(|X|^3), \quad A_2(X)=\frac12b_0x+O(|X|^3).
\]

Let us introduce
\[
b_2(X)=\frac12 X\cdot {\rm Hess}\,b(0)\cdot X.
\]
Thus, we have
\[
b(X)=b_0+b_2(X)+O(|X|^3).
\]

We have
\[
P^h_{D}=\sum_{1\leq \alpha,\beta\leq 2} g^{\alpha\beta}(X)
\nabla^h_\alpha\nabla^h_\beta+i h \sum_{1\leq \alpha\leq
2}\Gamma^\alpha \nabla^h_\alpha-hb(X)+h^{1/2} (b(X)-b_0),
\]
so its quadratic form is given by
\begin{multline*}
(P^h_{D}u,u)= \int_\Omega \sum_{1\leq \alpha,\beta\leq 2}
g^{\alpha\beta}(X) \nabla^h_\alpha u(X) \overline{\nabla^h_\beta
u(X)} \sqrt{g(X)}dX \\ -h\int_\Omega b(X)|u(X)|^2\sqrt{g(X)}dX
+h^{1/2} \int_\Omega (b(X)-b_0)|u(X)|^2\sqrt{g(X)}dX.
\end{multline*}
Note that
\[
P^h_{D}\geq 0.
\]

Now we move the operator $P_D^{h}$ into the Hilbert space
$L^2(\Omega,dX)$, using the unitary change of variables
$v=|g(X)|^{1/4}u$. For the corresponding operator
$$ \hat{P}_D^{h}=|g(X)|^{1/4}P_D^{h} |g(X)|^{-1/4}$$ in
$L^2(\Omega,dX)$, we obtain
\begin{multline*}
(\hat{P}_D^{h}v,v)=\int_\Omega \sum_{1\leq \alpha,\beta\leq 2}
g^{\alpha\beta}(X)\left( \nabla^h_\alpha-\frac14 ih
|g(X)|^{-1}\frac{\partial}{\partial X_\alpha}|g(X)|\right)
v(X)\times
\\ \times \overline{\left(\nabla^h_\beta -\frac14 ih
|g(X)|^{-1}\frac{\partial}{\partial X_\alpha}|g(X)|\right) v(X)}dX\\
-h\int_\Omega b(X)|v(X)|^2 dX +h^{1/2} \int_\Omega
(b(X)-b_0)|v(X)|^2 dX.
\end{multline*}

Put
\[
q(v)=\int_\Omega \sum_{1\leq \alpha\leq 2} \left|\left(
\nabla^h_\alpha-\frac14 ih |g(X)|^{-1}\frac{\partial}{\partial
X_\alpha}|g(X)|\right) v(X)\right|^2 dX.
\]

Then we have
\begin{multline}\label{e:qv}
|(\hat{P}_D^{h}v,v)-q(v)|\\ \leq \int_\Omega \sum_{1\leq \alpha\leq
2} |X|^2 \left|\left( \nabla^h_\alpha-\frac14 ih
|g(X)|^{-1}\frac{\partial}{\partial
X_\alpha}|g(X)|\right) v(X)\right|^2 dX\\
+C_1 h\int_\Omega |v(X)|^2 dX +C_2h^{1/2} \int_\Omega |X|^2 |v(X)|^2
dX.
\end{multline}

Consider the Dirichlet realization $P_{flat,D}^{h}$ of the operator
\[
P_{flat}^{h}=\left(i h \frac{\partial}{\partial x}-\frac12
b_0y\right)^2+\left(i h \frac{\partial}{\partial y}+ \frac12
b_0x\right)^2-hb_0+h^{1/2} b_2(X)
\]
in the space $L^2(\Omega,dX)$. So its quadratic form is given by
\[
(P_{flat}^{h}v,v)= q^{flat}(v)-hb_0\int_\Omega |v(X)|^2 dX +h^{1/2}
\int_\Omega b_2(X)|v(X)|^2 dX,
\]
where
\[
q^{flat}(v)=\int_\Omega \left|\left(i h \frac{\partial}{\partial
x}-\frac12 b_0y\right)v(X)\right|^2 dX + \int_\Omega \left|\left(i h
\frac{\partial}{\partial y}+ \frac12 b_0x\right)v(X)\right|^2 dX.
\]

So we have
\begin{multline}\label{e:P-P}
|(\hat{P}_D^{h}v,v)-(P_{flat}^{h}v,v)| \leq |q(v)-q^{flat}(v)|
\\ + \sum_\alpha \int_\Omega |X|^2 |\nabla^h_\alpha v(X)|^2 dX +
h\int_\Omega |X|^2 |v(X)|^2 dX +h^{1/2} \int_\Omega |X|^3|v(X)|^2
dX.
\end{multline}

Finally, we have
\begin{multline}\label{e:q-q}
|q(v)-q^{flat}(v)|\\ \leq C(q(v))^{1/2}\left[h\left(\int_\Omega
|X|^2 |v(X)|^2 dX\right)^{1/2}+\left(\int_\Omega |X|^6 |v(X)|^2
dX\right)^{1/2}\right].
\end{multline}

For a fixed $j\in \NN$, consider the subspace $V^{h,j}$ of
$L^2(\Omega,dx_g)$, generated by all eigenfunctions of the operator
$P^h_D$ associated to the eigenvalue $\lambda_\ell(P^h_D)$ with
$\ell=0,1,\ldots,j$. Thus, $V^{h,j}$ is a ($j+1$)-dimensional space
such that
\begin{equation}\label{e:Vhj}
(P^h_Du_h,u_h)\leq \lambda_j(P^h_D)\|u_h\|^2_{L^2(\Omega,dx_g)},
\quad u_h\in V^{h,j}.
\end{equation}
Moreover, by Lemma~\ref{l:weight}, for any real $k\geq 0$, there
exists $C_k>0$ such that, for any $u_h\in V^{h,j}$,
\begin{equation}\label{e:Xu}
\| |X|^ku_h\|_{L^2(\Omega,dx_g)}\leq C_{k}(h^{k/2}+h^{(3k+1)/8}) \|
u_h\|_{L^2(\Omega,dx_g)},
\end{equation}
for any $\alpha =1,2$ and any $k\geq 0$,
\begin{equation}\label{e:Xnablau}
\| |X|^k\nabla^h_\alpha u_h\|_{L^2(\Omega,dx_g)}\leq
C_{k}(h^{(k+1)/2}+h^{(3k+5)/8})\|u_h\|_{L^2(\Omega,dx_g)}.
\end{equation}

By \eqref{e:Xu} and \eqref{e:Xnablau}, for any real $k\geq 0$, there
exists $C_k>0$ such that, for any $v_h\in L^2(\Omega,dX)$ of the
form $v_h=|g(X)|^{1/4}u_h$ with $u_h\in V^{h,j}$,
\begin{equation}\label{e:Xv}
\| |X|^kv_h\|_{L^2(\Omega,dX)}\leq C_{k}(h^{k/2}+h^{(3k+1)/8}) \|
v_h\|_{L^2(\Omega,dX)},
\end{equation}
for any $\alpha =1,2$ and for any $k\geq 0$.
\begin{multline}\label{e:Xnablav}
\left( \int_\Omega |X|^{2k} \left|\left(i h \frac{\partial}{\partial
x}-\frac12 b_0y\right)v_h(X)\right|^2 dX\right)^{1/2} \\ +
\left(\int_\Omega |X|^{2k} \left|\left(i h \frac{\partial}{\partial
y}+ \frac12 b_0x\right)v_h(X)\right|^2 dX\right)^{1/2}\\ \leq
C_{k}(h^{(k+1)/2}+h^{(3k+5)/8})\|v_h\|_{L^2(\Omega,dX)}.
\end{multline}

The estimates \eqref{e:Xv} and \eqref{e:Xnablav} allow us show for
any $v_h\in L^2(\Omega,dX)$ of the form $v_h=|g(X)|^{1/4}u_h$ with
$u_h\in V^{h,j}$, first, using \eqref{e:Vhj}, \eqref{e:upperPh} and
\eqref{e:qv}, that,
\[
q(v_h)\leq Ch^{3/2}\|v_h\|_{L^2(\Omega,dX)}^2,
\]
next, using \eqref{e:q-q}, that
\[
|q(v_h)-q^{flat}(v_h)|\leq Ch^2\|v_h\|_{L^2(\Omega,dX)}^2,
\]
and finally, using \eqref{e:P-P}, that
\begin{equation}\label{e:flatRn}
(P_{flat}^{h}v_h,v_h)\leq
(\lambda_j(P^h_D)+C_jh^{15/8})\|v_h\|_{L^2(\Omega,dX)}^2.
\end{equation}

Let $\chi$ be a function from $C^\infty_c(\RR^n)$ such that ${\rm
supp}\,\chi\subset \Omega$ and $\chi\equiv 1$ in a neighborhood of
zero. By \eqref{e:Xv} and \eqref{e:Xnablav}, it follows that, for
any $k\in\NN$ there exists $C_k>0$ such that, for any $v_h\in
L^2(\Omega,dX)$ of the form $v_h=|g(X)|^{1/4}u_h$ with $u_h\in
V^{h,j}$,
\begin{multline}\label{e:chiv}
\|(1-\chi) v_h\|_{L^2(\Omega,dX)} + \|\frac{\partial \chi}{\partial
x} v_h\|_{L^2(\Omega,dX)}+ \|\frac{\partial \chi}{\partial y}
v_h\|_{L^2(\Omega,dX)}\\ \leq C_{k}h^{k}\| v_h\|_{L^2(\Omega,dX)}.
\end{multline}
and
\begin{multline}\label{e:chinablav}
\left( \int_\Omega (1-\chi(X)) \left|\left(i h
\frac{\partial}{\partial x}-\frac12 b_0y\right)v_h(X)\right|^2
dX\right)^{1/2} \\ + \left(\int_\Omega (1-\chi(X)) \left|\left(i h
\frac{\partial}{\partial y}+ \frac12 b_0x\right)v_h(X)\right|^2
dX\right)^{1/2}\\ \leq C_{k}h^k\|v_h\|_{L^2(\Omega,dX)}.
\end{multline}

Using \eqref{e:chiv} and \eqref{e:chinablav}, it is easy to check
that, for any $k>0$, there exists $C_{k}>0$ such that
\begin{equation}\label{e:flat-chi}
|(P_{flat}^{h}(\chi v_h), \chi v_h)-(P_{flat}^{h}v_h, v_h)| \leq C_k
h^k \|v_h\|^2.
\end{equation}

Consider the self-adjoint realization of the operator $P_{flat}^{h}$
in $L^2(\RR^2,dX)$. We will the same notation $P_{flat}^{h}$ for
this operator. Consider the $(j+1)$-dimensional subspace $W^{h,j}$
of $C^\infty_c(\RR^2)$, which consists of all functions $w_h\in
C^\infty_c(\RR^2)$ of the form $w_h=\chi |g(X)|^{1/4}u_h$ with
$u_h\in V^{h,j}$. By \eqref{e:flatRn} and \eqref{e:flat-chi}, it
follows that, for any $w_h\in W^{h,j}$,
\[
(P_{flat}^{h}w_h,w_h)\leq
(\lambda_j(P^h_D)+C_jh^{15/8})\|w_h\|_{L^2(\Omega,dX)}^2.
\]
By the mini-max principle, this immediately implies that, for any
$j>0$, there exists $C_{j}>0$ such that, for $j$-th eigenvalue
$\lambda_j(P_{flat}^{h})$ of $P_{flat}^{h}$, we have
\begin{equation}\label{e:flat-D-Rn}
\lambda_j(P_{flat}^{h}) \leq \lambda_j(P^h_D)+C_jh^{15/8}.
\end{equation}

It remains to recall that the eigenvalues of the Schr\"odinger
operator with constant magnetic field and positive quadratic
potential in $\RR^n$ can be computed explicitly. More precisely
(see, for instance, \cite[Theorem 2.2]{MatUeki}), the eigenvalues of
the operator
\[
H_{b,K}=\left(i \frac{\partial}{\partial x}-\frac12
by\right)^2+\left(i\frac{\partial}{\partial y}+ \frac12
bx\right)^2+\sum_{ij}K_{ij}X_iX_j
\]
are given by
\[
\lambda_{n_1n_2}=(2n_1+1)s_1+(2n_2+1)s_2,\quad n_1,n_2\in \NN,
\]
where
\[
s_1=\frac{1}{\sqrt{2}}\left[t_K+b^2-[(t_K+b^2)^2-4d_K]^{1/2}\right]^{1/2},
\]
\[
s_2=\frac{1}{\sqrt{2}}\left[t_K+b^2+[(t_K+b^2)^2-4d_K]^{1/2}\right]^{1/2},
\]
and
\[
t_K=\operatorname{Tr}K, \quad d_K=\det K.
\]
Applying this formula to the operator $P_{flat}^{h}$, we obtain that
its eigenvalues have the form:
\[
\lambda_{n_1n_2}=(2n_1+1)s_1+(2n_2+1)s_2-hb_0,\quad n_1,n_2\in \NN,
\]
where
\begin{align*}
s_1 & =\frac{h}{\sqrt{2}}\left[h^{1/2} t+b_0^2-[(h^{1/2}
t+b_0^2)^2-4h^{1/2} d]^{1/2}\right]^{1/2}\\
& = d^{1/2}b^{-1}_0h^{3/2}+O(h^{2}),
\end{align*}
and
\begin{align*}
s_2& =\frac{h}{\sqrt{2}}\left[h^{1/2} t+b_0^2+[(h^{1/2}
t+b_0^2)^2-4h^{1/2} d]^{1/2}\right]^{1/2}\\
& = hb_0+\frac12 t b^{-1}_0h^{3/2}+O(h^{2}).
\end{align*}
Thus, we obtain
\[
\lambda_{n_1n_2}=2n_2hb_0+ \left[n_1\frac{2d^{1/2}}{b_0}+n_2
\frac{t}{b_0}+\frac{a^2}{2b_0}\right]h^{3/2}+O(h^{2}),\quad
n_1,n_2\in \NN.
\]
For $j$th eigenvalue $\lambda_j(P_{flat}^{h})$ of $P_{flat}^{h}$, we
obtain
\begin{equation}\label{e:flat}
\lambda_j(P_{flat}^{h})= \left[\frac{2d^{1/2}}{b_0}j
+\frac{a^2}{2b_0}\right]h^{3/2}+O(h^{2}).
\end{equation}
Combining \eqref{e:flat-D-Rn} and \eqref{e:flat}, we immediately
complete the proof of Theorem~\ref{t:PAV}.

\begin{proof}[Proof of Theorem \ref{t:mainbis}]
Fix $j\in \NN$. By Theorem~\ref{t:main}, there exist $C>0$ and
$h_0>0$ such that, for any $h\in (0,h_0]$,
\[
I_j\; \bigcap \; {\rm Spec}(H^h) =\{\lambda_j(H^h)\},
\]
where
\[
I_j=\left(hb_0 +h^2\left[\frac{2d^{1/2}}{b_0}j+
\frac{a^2}{2b_0}\right]-Ch^{19/8}, hb_0
+h^2\left[\frac{2d^{1/2}}{b_0}j+
\frac{a^2}{2b_0}\right]+Ch^{5/2}\right).
\]
On the other hand, by Corollary~\ref{c:dist}, for any natural $N$,
there exist $C^\prime>0$ and $h^\prime_{0} >0$ such that, for any
$h\in (0,h^\prime_{0}]$,
\[
{\rm dist}(\mu_{j0N}^h, {\rm Spec}(H^h))\leq C^\prime
h^{\frac{N+3}{2}}.
\]
Without loss of generality, we can assume that, for any $h\in
(0,\min(h_0, h^\prime_{0})]$,
\[
(\mu_{j0N}^h-C^\prime h^{\frac{N+3}{2}},\mu_{j0N}^h+C^\prime
h^{\frac{N+3}{2}})\cap I_\ell=\emptyset, \forall \ell\neq j.
\]
Hence, for any $h\in (0,\min(h_0, h^\prime_{0})]$, $\lambda_j(H^h)$
is the point of ${\rm Spec}(H^h)$, closest to $\mu_{j0N}^h$. It
follows that
\[
|\lambda_j(H^h)-\mu_{j0N}^h|\leq C^\prime h^{\frac{N+3}{2}}, \quad
h\in (0,\min(h_0, h^\prime_{0})]\,,
\]
that proves \eqref{e:mainbis} with $\alpha_{j,\ell}=\mu_{j,0,\ell}$.
\end{proof}

\section{Periodic case and spectral gaps}\label{s:gaps}
In this Section, we apply the results of Section~\ref{s:upper} to
the problem of existence of gaps in the spectrum of a periodic
magnetic Schr\"odinger operator. Some related results on spectral
gaps for periodic magnetic Schr\"odinger operators can be found in
\cite{BDP,gaps,HS88,HSLNP345,HempelHerbst95,HempelPost02,HerbstNakamura,KMS,Ko04,bedlewo,MS,Nakamura95}
(see also the references therein).

Let $ M$ be a two-dimensional noncompact oriented manifold of
dimension $n\geq 2$ equipped with a properly discontinuous action of
a finitely generated, discrete group $\Gamma$ such that $M/\Gamma$
is compact. Suppose that $H^1(M, \RR) = 0$, i.e. any closed $1$-form
on $M$ is exact. Let $g$ be a $\Gamma$-invariant Riemannian metric
and $\bf B$ a real-valued $\Gamma$-invariant closed 2-form on $M$.
Assume that $\bf B$ is exact and choose a real-valued 1-form $\bf A$
on $M$ such that $d{\bf A} = \bf B$. Write ${\bf B}=b dx_g$, where
$b\in C^\infty(M)$ and $dx_g$ is the Riemannian volume form. Let
\[
b_0=\min_{x\in M}b(x).
\]
Assume that there exist a (connected) fundamental domain $\cF$ and a
constant $\epsilon_0>0$ such that
\[
b(x) \geq b_0+\epsilon_0, \quad x\in \partial\cF.
\]

We will consider the magnetic Schr\"odinger operator $H^h$ as an
unbounded self-adjoint operator in the Hilbert space $L^2(M)$. Using
the results of \cite{diff2006}, one can immediately derive from
Theorem~\ref{t:qmodes} the following result on existence of gaps in
the spectrum of $H^h$ in the semiclassical limit.

We will use the above notation
\[
t={\rm Tr}\left(\frac12 {\rm Hess}\,b(x_0)\right), \quad d=\det
\left(\frac12 {\rm Hess}\,b(x_0)\right).
\]
For any $k\in \NN$, put
\[
c_k=(2k+1)\frac{d^{1/2}}{b_0}\\
+(2k^2+2k+1)\frac{t}{2b_0}+\frac{1}{2}(k^2+k) R(x_0).
\]

\begin{thm}
Assume that $b_0>0$ and there exist $x_0\in \cF$ and $C>0$ such that
for all $x$ in some neighborhood of $x_0$ the estimates hold:
\[
C^{-1}\, d(x,x_0)^2\leq b(x)-b_0 \leq C \, d(x,x_0)^2\,.
\]
Then, for any natural $k$ and $N$, there exist $C_{k,N}>c_k$ and
$h_{k,N}>0$ such that the spectrum of $H^h$ in the interval
\[
[(2k+1) hb_0+ h^2c_k,(2k+1) hb_0+ h^2C_{k,N}]
\]
has at least $N$ gaps for any $h\in (0, h_{k,N}]$.
\end{thm}

\end{document}